\documentclass[10pt,a4paper]{article}
\usepackage[dvips]{color}
\usepackage{indentfirst,latexsym,bm}
\usepackage{graphicx}
\usepackage{amsmath,amsthm}
\usepackage{amssymb}
\usepackage{multicol}
\usepackage{ulem}
\usepackage{amsfonts}
\usepackage{mathrsfs}
 \newtheorem{thm}{Theorem}[section]
 \newtheorem{cor}[thm]{Corollary}
 \newtheorem{lem}[thm]{Lemma}
 \newtheorem{prop}[thm]{Proposition}
 
 \theoremstyle{definition}
 \newtheorem{defn}[thm]{Definition}
 \newtheorem{exmp}{Example}
 
 \newtheorem{rem}{Remark}

 \DeclareMathAlphabet{\mathsfsl}{OT1}{cmss}{m}{sl}
 \newcommand{\Enum}{\mathbb{E}}
 \newcommand{\Rnum}{\mathbb{R}}
 
 \newcommand{\Znum}{\mathbb{Z}}
 \newcommand{\Nnum}{\mathbb{N}}

  \newcommand{\diag}{\mathrm{diag}}
 \newcommand{\tensor}[1]{\mathsf{#1}}
 
 \newcommand{\set}[1]{\left\{#1\right\}}

\title{On the reversibility of the observed process of three-state hidden Markov model}

\author{\rm\small
\noindent CHEN Yong\\
\noindent \footnotesize School of Mathematics and Computing Science, Hunan
University of Science and Technology,\\
\noindent \footnotesize Xiangtan, Hunan, 411201,
P.R.China. chenyong77@gmail.com}
\date{}
\begin{document}
%-------------------------------------------------------------------------------------------------------
%-------------------------------------------------------------------------------------------------------
\maketitle
%-------------------------------------------------------------------------------------------------------
\maketitle \noindent {\bf Abstract } \\
For the continuous-time and the discrete-time three-state hidden
Markov model, the flux of the likelihood function up to
3-dimension of the observed process is shown explicitly. As an application, the sufficient and necessary condition of
the reversibility of the observed process is shown. \\
\noindent { \bf keywords: } hidden Markov models;
likelihood function(joint probability distribution); reversibility.

\maketitle

%-------------------------------------------------------------------------------------------------------
\section{Introduction}
There has been a large amount of literature published on the time
reversibility in probability, such as References
\cite{kol,hp,dsf,qqt,jdq,Ser,GJQ,wq}, which are
mainly about the Markov processes and the semi-Markov processes
(or Markov renewal processes). In the case of a Markov process
with finite state (discrete time or continuous time), Kolmogorov's
criterion for time reversibility is well-known. In the References~\cite{Wes,hlp,Cheng,TangZhang,TangC}, they
examined time reversibility in the context of a univariate
stationary linear time series (Gaussian or non-Gaussian) and of
multivariate linear processes.

In Reference\cite{mz}, for the hidden Markov model, it is shown
that the reversibility of the observed process is not equivalent
to that of the underlying Markov chain, i.e., if the underlying
Markov chain is reversible, then the observed process is
reversible too, however, if the Markov chain is irreversible, then
the observed process is either reversible or irreversible. In
Reference~\cite{BM}, the necessary and sufficient conditions for
reversibility of hidden Markov chains on general (countable)
spaces are obtained, however, the reversibility therein is
concerning the complete process, i.e., the bivariate stochastic
process containing both the underlying process and the observed
process. That is to say, the above two types of reversibility of
the hidden Makov model are different completely.

In the present paper, for continuous-time three-state Markov processes, we
calculate the flux of the likelihood function (joint probability
distribution) of  the observed process. As an application, the
sufficient and necessary condition of the reversibility of the
three-state hidden Markov model is shown (in the sense of
\cite{mz}, not in the sense of \cite{BM}). In fact, besides the
reversibility of the underlying Markov process, the reversibility
of the observed process is distinguished by whether the state-dependent probability matrix is regular (Definition~\ref{regular}, Theorem~\ref{dingl}). We illustrate by an example
that one cannot detect irreversibility in some cases by comparing
directional moments like that used in \cite[p104]{mz}.

We have also investigated the discrete-time three-state hidden Markov model. Since the method
is similar to the continuous-time case, we list the conclusions in Appendix (Section~4) and omit most of the proofs.
For the discrete-time case, the reversibility is also related to whether zero is an eigenvalue of the 1-step
transition probability matrix (Proposition~\ref{propzero}). Here we see a difference between discrete-time
and continuous-time hidden Markov model.

The reversibility of the hidden Markov model may be of interest in some biological studies. An approach to modelling the DNA
sequence is to use a hidden Markov model; see, for example,
Reference \cite{ch,de}. Since DNA sequences have directions, we
should rule out the reversible hidden Markov model.
%-------------------------------------------------------------------------------------------------------

%-------------------------------------------------------------------------------------------------------
\section{The flux of the likelihood function}\label{section2}
%-------------------------------------------------------------------------------------------------------
Let $\set{S_t: \,t\in \Rnum^+}$ be the observed process with state
space $\mathcal{S}=\set{0,1,2,\cdots,K-1}$.
%-------------------------------------------------------------------------------------------------------
\begin{defn}
The $n$-dimension likelihood function of $\set{S_t: \,t\in \Rnum^+}$ is defined
as $\Pr(S_{t_1}=s_1,\,S_{t_2}=s_2,\,\cdots,\,S_{t_n}=s_n)$, where
$n\in \Nnum$ and $0\leqslant t_1 \leqslant t_2 \leqslant
\cdots\leqslant t_n$.
The flux of the likelihood function of $\set{S_t: \,t\in \Rnum^+}$ is defined
as
\begin{equation}
 \hspace{-20mm} \Pr(S_{t_1}=s_1,\,S_{t_2}=s_2,\,\cdots,\,S_{t_n}=s_n)-\Pr(S_{t_1^-}=s_1,\,S_{t_2^-}=s_2,\,\cdots,\,S_{t_n^-}=s_n),
\end{equation}
where $t_k^-=t_1+t_n-t_k$.
\end{defn}
%-------------------------------------------------------------------------------------------------------

Let $\set{C_t: \,t\in \Rnum^+}$ be an irreducible three-state
Markov process with the transition rate matrix $\tensor{Q}$, and
the stationary distribution $\mu=(\mu_1,\mu_2,\mu_3)$, where
$\mu_1+\mu_2+\mu_3=1,\,\mu_i>0$.

%-------------------------------------------------------------------------------------------------------
\begin{equation}\label{matrix original}
\tensor{Q}= \left[
\begin{array}{lll}
     -a_1&a_2 &a_3\\
     b_1&-b_2&b_3\\
     c_1&c_2 &-c_3
\end{array}
\right ],
\end{equation}
%-------------------------------------------------------------------------------------------------------
where
$a_1=a_2+a_3,\,b_2=b_1+b_3,\,c_3=c_1+c_2,\,a_i,\,b_i,\,c_i\geqslant
0,\,i=1,2,3$, and $a_1b_2c_3,\,b_1+c_1,\,a_2+c_2,\,a_3+b_3>0$
(i.e. irreducible). By the stationarity, it is clear that the
transition rate flux of $\set{C_t: \,t\in \Rnum^+}$ is
%-------------------------------------------------------------------------------------------------------
\begin{equation}\label{flux}
  \mu_1a_2-\mu_2b_1=\mu_2 b_3 - \mu_3 c_2 = \mu_3 c_1 - \mu_1 a_3.
\end{equation}
%-------------------------------------------------------------------------------------------------------
Let $\nu=\mu_1a_2-\mu_2b_1$. And the eigen-equation of $\tensor{Q}$ is
%-------------------------------------------------------------------------------------------------------
\begin{equation}\label{eigenequ}
    \lambda(\lambda^2 + \alpha \lambda +\beta)=0.
\end{equation}
%-------------------------------------------------------------------------------------------------------
Denote by $-\lambda_1,\,-\lambda_2$ the nonzero eigenvalues of
$\tensor{Q}$. Let $\triangle =\alpha^2 - 4\beta$.
%-------------------------------------------------------------------------------------------------------

%-------------------------------------------------------------------------------------------------------
Similar to the
denotation of Reference~\cite{mrh}, let $S_1^j$ and $T_1^j$ denote
the sequence from $1$ to $j$  of the observed states and
observation times. The Markov assumption for the hidden process is
given by
%-------------------------------------------------------------------------------------------------------
\begin{equation}
 \begin{array}{ll}
 \Pr[C(t_j)\,|\,C(t_1),C(t_2),\dots,C(t_{j-1}),S_1^{j-1},T_1^j=t_1^j]\\
 =\Pr[C(t_j)\,|\,C(t_{j-1}),T_{j-1}^j=t_{j-1}^j]\\
 =\tensor{P}_{c_{j-1},c_j}(t_j-t_{j-1}),
 \end{array}
\end{equation}
%-------------------------------------------------------------------------------------------------------
where the quantity $\tensor{P}_{c_{j-1},c_j}(t_j-t_{j-1})$ denotes
the transition probability of occupying sate $c_j$ at time
$T_j=t_j$ given that the process was in state $c_{j-1}$ at
$t_{j-1}$. As indicated by the last equality, the transition
probabilities of the process are assumed to be time homogeneous.
We also assume that, conditional on the state of the hidden
process at time $t_j$, an observation $S_j$ is independent of all
previous observations and the hidden process prior to time $t_j$:
%-------------------------------------------------------------------------------------------------------
\begin{equation}
 \begin{array}{ll}
  \Pr[S_j\,|\,C(t_1),C(t_2),\dots,C(t_{j}),S_1^{j-1},T_1^j=t_1^j]\\
=\Pr[S_j\,|\,C(t_{j}),T_j=t_j]\\
=\pi({s_j\,|\,c_j}).
\end{array}
\end{equation}
%-------------------------------------------------------------------------------------------------------

%-------------------------------------------------------------------------------------------------------
Let the `state-dependent probability' (i.e., emission probability,
signal probability) matrix be $\tensor{\Pi}=(\pi({k\,|\,i})),\,
i=1,2,3;\,k=0,1,2,\dots,K-1$ (i.e., a $3\times K$ matrix). Note
that the rows of $\tensor{\Pi}$ must sum to 1 \footnote{The
state-dependent probability here is the transpose matrix of that
in Reference \cite{mz}.}.
%-------------------------------------------------------------------------------------------------------
Let $\varphi_{k}$ be the $k$-column of $\tensor{\Pi}$ and
$\tensor{\wedge}_k=\diag\set{\varphi_k},\,k=0,1,2,\dots,K-1$.
%-------------------------------------------------------------------------------------------------------
\begin{prop}\label{prop1}
The flux of the 2-dimension likelihood function is when
  $t>0$,
%-------------------------------------------------------------------------------------------------------
\begin{eqnarray*}
 \Pr\set{S_0=i,S_t=j}-\Pr\set{S_0=j,S_t=i}=\frac{\nu A}{\lambda_1-\lambda_2}[e^{-\lambda_2 t}-e^{-\lambda_1 t}],
\end{eqnarray*}
%-------------------------------------------------------------------------------------------------------
where $\nu$ is the transition rate flux,
$A=(y_2-x_2)(x_1-z_1)-(x_2-z_2)(y_1-x_1)$,
$(x_1,y_1,z_1)'=\varphi_i$, and $ (x_2,y_2,z_2)'=\varphi_j$.
\end{prop}
%-------------------------------------------------------------------------------------------------------
\begin{cor}\label{cor2}
If the rank of the state-dependent probability is $1$ or $2$, then
$\Pr\set{S_0=i,S_t=j}-\Pr\set{S_0=j,S_t=i}=0$.
\end{cor}
%-------------------------------------------------------------------------------------------------------
\begin{thm}\label{3dim likely}
The flux of the following 3-dimension likelihood function is when
  $r,\,t>0$,
%-------------------------------------------------------------------------------------------------------
\begin{eqnarray*}
 & \Pr\set{S_0=S_r=S_{r+t}=i}- \Pr\set{S_0=S_t=S_{t+r}=i}\\
 &=\frac{\nu D}{\lambda_1-\lambda_2}(e^{-\lambda_2 r-\lambda_1 t}-e^{-\lambda_2 t-\lambda_1 r}),
\end{eqnarray*}
%-------------------------------------------------------------------------------------------------------
where $\nu$ is the transition rate flux, $D=(x-y)(y-z)(z-x)$,
$(x,y,z)'=\varphi_i$.
\end{thm}
%-------------------------------------------------------------------------------------------------------
Proofs of Proposition~\ref{prop1}, Corollary~\ref{cor2}, Theorem~\ref{3dim likely} are presented in Subsection~\ref{proofs0}.
%-------------------------------------------------------------------------------------------------------

%-------------------------------------------------------------------------------------------------------
\subsection{Proofs }\label{proofs0}
%-------------------------------------------------------------------------------------------------------
Let
$\tensor{U}=\diag\set{\mu_1,\mu_2,\mu_3}$. Then
%-------------------------------------------------------------------------------------------------------
\begin{equation}\label{matrix1}
\tensor{U Q}-\tensor{Q}'\tensor{U}= \nu\left[
\begin{array}{lll}
     0&1 &-1\\
     -1&0&1\\
     1&-1 &0
\end{array}
\right ]
\end{equation}
%-------------------------------------------------------------------------------------------------------
Let $\vec{e}=(1,1,1)'$ and the matrix $\tensor{L}=\vec{e}\mu$.
%-------------------------------------------------------------------------------------------------------
\begin{lem}\label{probability}
If $\triangle \neq 0$, then for $t\in \Rnum^+ $, the $t$-step
transition probability matrix is
%-------------------------------------------------------------------------------------------------------
\begin{equation}
  \tensor{P}(t)=g_t\tensor{L}+d_t\tensor{Q}+f_t\tensor{I},
\end{equation}
%-------------------------------------------------------------------------------------------------------
where
%-------------------------------------------------------------------------------------------------------
\begin{equation}\label{dk1}
\begin{array}{ll}
d_t  = \frac{e^{-\lambda_2 t}-e^{-\lambda_1 t}}{\lambda_1-\lambda_2},\\
f_t =\frac
{\lambda_1e^{-\lambda_2 t}-\lambda_2e^{-\lambda_1 t}}{\lambda_1-\lambda_2},\\
g_t  = 1-f_t.
\end{array}
\end{equation}
%-------------------------------------------------------------------------------------------------------
\end{lem}
%-------------------------------------------------------------------------------------------------------
It is Proposition~4.3 of Reference~\cite{cy}. The reader can also
refer to Theorem~14.9 of Reference~\cite{CCL}. To write it in
terms of function with matrix coefficients is the key to the
results in the present paper.
%-------------------------------------------------------------------------------------------------------

%-------------------------------------------------------------------------------------------------------
\begin{rem}\label{remark}
 Fix the value of $\lambda_1$, and let $\lambda_2\rightarrow
 \lambda_1$, then one has
%-------------------------------------------------------------------------------------------------------
\begin{equation}\label{dk2}
\begin{array}{ll}
d_t  = t\, e^{-\lambda_1 t},\\
f_t =(1+\lambda_1 t)e^{-\lambda_1 t},\\
g_t = 1-f_t.
\end{array}
\end{equation}
%-------------------------------------------------------------------------------------------------------
It is exactly the $t$-step transition probability matrix in the
case $\triangle = 0$, please refer to \cite{cy,CCL} and the
references therein. That is to say, Eq.(\ref{dk2}) is the same as
Eq.(\ref{dk1}) in the sense of limit. We therefore do not
distinguish whether $\triangle = 0$ or not for all the subsequent
formulas.
\end{rem}
%-------------------------------------------------------------------------------------------------------
\noindent{\it Proof of Proposition~\ref{prop1}.\,}
%-------------------------------------------------------------------------------------------------------
Since $\tensor{L}=\vec{e}\mu$, we have
  $$\mu\tensor{\wedge}_i\tensor{L}\tensor{\wedge}_j\vec{e}=(\mu\tensor{\wedge}_i\vec{e})(\mu\tensor{\wedge}_j\vec{e})
  =(\mu\tensor{\wedge}_j\vec{e})(\mu\tensor{\wedge}_i\vec{e})=\mu\tensor{\wedge}_j\tensor{L}\tensor{\wedge}_i\vec{e}.$$
Note that
$\tensor{\wedge}_i\tensor{\wedge}_j=\tensor{\wedge}_j\tensor{\wedge}_i$.
  By Eq.(2.27) of Reference~\cite{mz} and Lemma~\ref{probability}, we have
%-------------------------------------------------------------------------------------------------------
\begin{eqnarray*}
  && \Pr\set{S_0=i,S_t=j}-\Pr\set{S_0=j,S_t=i}\\
  &=&\mu\tensor{\wedge}_i\tensor{P}(t)\tensor{\wedge}_j\vec{e}-\mu\tensor{\wedge}_j\tensor{P}(t)\tensor{\wedge}_i\vec{e}\\
  &=&\mu\tensor{\wedge}_i(g_t\tensor{L}+d_t\tensor{Q}+f_t\tensor{I})\tensor{\wedge}_j\vec{e}-\mu\tensor{\wedge}_j(g_t\tensor{L}+d_t\tensor{Q}+f_t\tensor{I})\tensor{\wedge}_i\vec{e}\\
  &=&g_t(\mu\tensor{\wedge}_i\tensor{L}\tensor{\wedge}_j\vec{e}-\mu\tensor{\wedge}_j\tensor{L}\tensor{\wedge}_i\vec{e})
  +d_t(\mu\tensor{\wedge}_i\tensor{Q}\tensor{\wedge}_j\vec{e}-\mu\tensor{\wedge}_j\tensor{Q}\tensor{\wedge}_i\vec{e})
  +f_t(\mu\tensor{\wedge}_i\tensor{\wedge}_j\vec{e}-\mu\tensor{\wedge}_j\tensor{\wedge}_i\vec{e})\\
  &=&d_t(\mu\tensor{\wedge}_i\tensor{Q}\tensor{\wedge}_j\vec{e}-\mu\tensor{\wedge}_j\tensor{Q}\tensor{\wedge}_i\vec{e}).
\end{eqnarray*}
%-------------------------------------------------------------------------------------------------------
Since
$\mu\tensor{\wedge}_i\tensor{Q}\tensor{\wedge}_j\vec{e}=\varphi_i'\tensor{U
Q}\varphi_j$, we have
$$\mu\tensor{\wedge}_j\tensor{Q}\tensor{\wedge}_i\vec{e}=\varphi_j'\tensor{U
Q}\varphi_i=(\varphi_j'\tensor{U Q}\varphi_i)'=\varphi_i'\tensor{
Q}'\tensor{U}\varphi_j.$$ By Eq.(\ref{matrix1}), we have
%-------------------------------------------------------------------------------------------------------
\begin{equation}\label{equation6}
\mu\tensor{\wedge}_i\tensor{Q}\tensor{\wedge}_j\vec{e}-\mu\tensor{\wedge}_j\tensor{Q}\tensor{\wedge}_i\vec{e}
  =\varphi_i'(\tensor{U Q}-\tensor{Q}'\tensor{U})\varphi_j
  = \nu A.
\end{equation}
This ends the proof.\hfill $\Box$
%-------------------------------------------------------------------------------------------------------
\\
\\
%-------------------------------------------------------------------------------------------------------
\noindent{\it Proof of Corollary \ref{cor2}.\,}
%-------------------------------------------------------------------------------------------------------
Note that
$A=\det\set{\tensor{H}}$, where ``det" is the determinant
function, and $\tensor{H}=[\vec{e},\,\varphi_i,\,\varphi_j]$ is a $3\times 3$ matrix.
If the rank
of the state-dependent probability is $1$, then $\varphi_i,\,\varphi_j$ are linear dependent and we obtain that $A=0$.
If
the rank of the state-dependent probability is $2$, and if
$\varphi_i,\,\varphi_j$ are linear independent, then they are one
base of $\set{\varphi_k, k=0,1,2,\dots,K-1}$. Note that
$\vec{e}=\sum_{k=0}^{K-1}\varphi_i$. Then
$\varphi_i,\,\varphi_j,\,\vec{e}$ are linear dependent. Thus
$A=0$. This ends the proof.\hfill $\Box$
%-------------------------------------------------------------------------------------------------------
\\
\\
%-------------------------------------------------------------------------------------------------------
\noindent{\it Proof of Theorem \ref{3dim likely}.\,}
%-------------------------------------------------------------------------------------------------------
Since $\tensor{L}=\vec{e}\mu$, we have
\begin{eqnarray*}
\mu\tensor{\wedge}_i\tensor{L}\tensor{\wedge}_i\tensor{Q}\tensor{\wedge}_i\vec{e}
  =(\mu\tensor{\wedge}_i\vec{e})(\mu\tensor{\wedge}_i\tensor{Q}\tensor{\wedge}_i\vec{e})
  =(\mu\tensor{\wedge}_i\tensor{Q}\tensor{\wedge}_i\vec{e})(\mu\tensor{\wedge}_i\vec{e})
  =\mu\tensor{\wedge}_i\tensor{Q}\tensor{\wedge}_i\tensor{L}\tensor{\wedge}_i\vec{e}.
\end{eqnarray*}
Similar to Proposition~\ref{prop1}, we obtain
%-------------------------------------------------------------------------------------------------------
\begin{eqnarray}
   & & \Pr\set{S_0=S_r=S_{r+t}=i}- \Pr\set{S_0=S_t=S_{t+r}=i}\nonumber\\
  &=&\mu\tensor{\wedge}_i\tensor{P}(r)\tensor{\wedge}_i\tensor{P}(t)\tensor{\wedge}_i\vec{e}-\mu\tensor{\wedge}_i\tensor{P}(t)\tensor{\wedge}_i\tensor{P}(r)\tensor{\wedge}_i\vec{e} \nonumber\\
  &=&\mu\tensor{\wedge}_i(g_r\tensor{L}+d_r\tensor{Q}+f_r\tensor{I})\tensor{\wedge}_i(g_t\tensor{L}+d_t\tensor{Q}+f_t\tensor{I})\tensor{\wedge}_i\vec{e}\nonumber\\
  &-&\mu\tensor{\wedge}_i(g_t\tensor{L}+d_t\tensor{Q}+f_t\tensor{I})\tensor{\wedge}_i(g_r\tensor{L}+d_r\tensor{Q}+f_r\tensor{I})\tensor{\wedge}_i\vec{e}\nonumber\\
  &=&(g_rd_t-g_td_r)(\mu\tensor{\wedge}_i\tensor{L}\tensor{\wedge}_i\tensor{Q}\tensor{\wedge}_i\vec{e}-\mu\tensor{\wedge}_i\tensor{Q}\tensor{\wedge}_i\tensor{L}\tensor{\wedge}_i\vec{e})
  +(g_rf_t-g_tf_r)(\mu\tensor{\wedge}_i\tensor{L}\tensor{\wedge}_i^2\vec{e}-\mu\tensor{\wedge}_i^2\tensor{L}\tensor{\wedge}_i\vec{e})\nonumber\\
  &+&(d_rf_t-d_tf_r)(\mu\tensor{\wedge}_i\tensor{Q}\tensor{\wedge}_i^2\vec{e}-\mu\tensor{\wedge}_i^2\tensor{Q}\tensor{\wedge}_i\vec{e})\nonumber\\
  &=&(d_rf_t-d_tf_r)(\mu\tensor{\wedge}_i\tensor{Q}\tensor{\wedge}_i^2\vec{e}-\mu\tensor{\wedge}_i^2\tensor{Q}\tensor{\wedge}_i\vec{e}).\nonumber
\end{eqnarray}
%-------------------------------------------------------------------------------------------------------
It follows from Lemma~\ref{probability} that
%-------------------------------------------------------------------------------------------------------
\begin{eqnarray*}
  d_rf_t-d_tf_r=\frac{1}{\lambda_1-\lambda_2}[e^{-\lambda_2 r-\lambda_1 t}-e^{-\lambda_2 t-\lambda_1 r}].
\end{eqnarray*}
%-------------------------------------------------------------------------------------------------------
Let ${\psi}=(x^2,y^2,z^2)'$. Similar to Eq.(\ref{equation6}), we
have
\begin{equation}\label{matrix00}
\begin{array}{ll}
  \quad \mu\tensor{\wedge}_i\tensor{Q}\tensor{\wedge}_i^2\vec{e}-\mu\tensor{\wedge}_i^2\tensor{Q}\tensor{\wedge}_i\vec{e}\\
  =\varphi_i'(\tensor{U Q}-\tensor{Q}'\tensor{U})\psi\\
  =\nu[x(y^2-z^2)+y(z^2-x^2)+z(x^2-y^2)]\\
  = \nu D.
\end{array}
\end{equation}
This ends the proof.
\hfill $\Box$
%-------------------------------------------------------------------------------------------------------

%-------------------------------------------------------------------------------------------------------
\section{The reversibility of the observed process}\label{section4}
The observed process is the same as in Section~\ref{section2}.
%-------------------------------------------------------------------------------------------------------
\begin{defn}
The observed process is said to be reversible if its
finite-dimensional distributions are invariant under reversal of
time, i.e., the flux of the likelihood function vanishes,
$$\Pr(S_{t_1}=s_1,\,S_{t_2}=s_2,\,\cdots,\,S_{t_n}=s_n)
-\Pr(S_{t_1^-}=s_1,\,S_{t_2^-}=s_2,\,\cdots,\,S_{t_n^-}=s_n)=0,$$
where $t_k^-=t_1+t_n-t_k$, for all positive integers $n$ and all
$0\leqslant t_1 \leqslant t_2 \leqslant \cdots\leqslant t_n$.
\end{defn}
%-------------------------------------------------------------------------------------------------------

%-------------------------------------------------------------------------------------------------------
\begin{defn}\label{regular}
Two rows of a matrix are said to be equal if they are two equal
vectors. If any two rows of the state-dependent probability matrix
$\tensor{\Pi}$ are not equal, we say that $\tensor{\Pi}$ is
regular. Otherwise, we say that $\tensor{\Pi}$ is singular, i.e.,
there are at least two undistinguishable states among the three
hidden states by means of observation.
\end{defn}
%-------------------------------------------------------------------------------------------------------
\begin{thm}\label{kaishi}
 If the underlying Markov process is reversible, then the observed process is reversible too.
\end{thm}
%-------------------------------------------------------------------------------------------------------
Although the proof in Reference \cite[P102-103]{mz} is about the
discrete-time hidden Markov model, it is still valid for the
continuous-time one and is ignored here.
%-------------------------------------------------------------------------------------------------------

%-------------------------------------------------------------------------------------------------------
\begin{thm}\label{dingl}
The observed process of the continuous-time three-state hidden
Makov model is irreversible, if and only if the underlying Markov
process is irreversible and the state-dependent probability matrix
is regular.
\end{thm}
%-------------------------------------------------------------------------------------------------------
Proof of Theorem~\ref{dingl} is presented in Section~\ref{subsec}.
%-------------------------------------------------------------------------------------------------------

If $\tensor{\Pi}$ is regular, the rank of $\tensor{\Pi}$ is $3$ or
$2$. If $\tensor{\Pi}$ is singular, the rank of $\tensor{\Pi}$ is
$2$ or $1$. That is to say, the rank of $\tensor{\Pi}$ is involved
in the reversibility of the observed processes.
%-------------------------------------------------------------------------------------------------------

By Reference~\cite{kem}, the hard limiting (or clipping)
transformation is very useful from a practical viewpoint, and the
rhythm inherited in the binary series carries a great deal of
information about the original series. If we maintain the
regularity condition of $\tensor{\Pi}$ when clipping the observed
process of the hidden Markov model, it preserves the
time-reversibility property by the last theorem (e.g.,
Example~\ref{exmp2}).
%-------------------------------------------------------------------------------------------------------

The time reversibility of high-order hidden Markov models (e.g.,
more than four-state) is difficult to be solved completely. We can
only find some simple sufficient conditions of the
irreversibility, for example, the underlying Markov process is
irreversible and the rank of the state-dependent probability
matrix is equal to the number of states of the underlying Markov
process (similar to Proposition~\ref{prop2}).
%-------------------------------------------------------------------------------------------------------
\begin{rem}
Let the complete process be $\set{X_t=(C_t,\,S_t),\,t\geqslant
0}$. Clearly, it is still be a finite-state Markov process, please
refer to \cite{bp}. Similar to Theorem~2.1 of Reference~\cite{BM},
by Kolmogorov's criterion, we can show that the complete process
is reversible if and only if the underlying process is reversible.
That is to say, there are two different types of reversibility of
the hidden Makov model.
\end{rem}
%-------------------------------------------------------------------------------------------------------
\subsection{Applications}
%-------------------------------------------------------------------------------------------------------
\begin{exmp}\label{exmp1}
The deterministic function of a Markov process is a special case
of hidden Markov model. Let $f=(f_1,\,f_2,\,f_3)'$ be a function
defined on the state space. If $f=(1,1,0)'$ or $f=(1,0,0)'$ like
that used in Reference~\cite{qe}, then the state-dependent
probability matrices are respectively
%-------------------------------------------------------------------------------------------------------
\begin{equation*}
\tensor{\Pi}_1= \left[
\begin{array}{ll}
     0&1 \\
     0&1\\
     1&0
\end{array}
\right ], \qquad \quad \tensor{\Pi}_2= \left[
\begin{array}{ll}
     0&1 \\
     1&0\\
     1&0
\end{array}
\right ].
\end{equation*}
%-------------------------------------------------------------------------------------------------------
Since the state-dependent probability matrices are singular, the
observed processes is reversible by Theorem~\ref{dingl}.
\end{exmp}
%-------------------------------------------------------------------------------------------------------
\begin{exmp}\label{exmp2}
Suppose $\set{C_t,\,t\geqslant 0}$ be the irreversible Markov
process with transition rate matrix
%-------------------------------------------------------------------------------------------------------
\begin{equation}\label{matrix q}
\tensor{Q}= \left[
\begin{array}{lll}
      -2/3&1/3 &1/3\\
     2/3&-1&1/3\\
     1/2&1/2 &-1
\end{array}
\right ].
\end{equation}
%-------------------------------------------------------------------------------------------------------
Let $\set{S_t}$\footnote{$\set{S_t}$ comes from the example in
Reference~\cite[p105]{mz}.}, $\set{\xi_t},\,\set{\eta_t}$ be three
observed processes with state-dependent probability matrices
respectively
%-------------------------------------------------------------------------------------------------------
\begin{equation*}
 \hspace{-12mm}\tensor{\Pi}_1= \left[
\begin{array}{lll}
     1&0&0 \\
     1/4&1/2&1/4\\
     0&0&1
\end{array}
\right ], \quad  \tensor{\Pi}_2= \left[
\begin{array}{lll}
     1&0&0 \\
     1/4&1/2&1/4\\
     1/2&1/3&1/6
\end{array}
\right ],\quad \tensor{\Pi}_3= \left[
\begin{array}{ll}
     1&0\\
     1/4&3/4\\
     0&1
\end{array}
\right ].
\end{equation*}
%-------------------------------------------------------------------------------------------------------
Since all the state-dependent probability matrices are regular,
the observed processes are irreversible by Theorem~\ref{dingl}.
$\set{\eta_t}$ is clipped from $\set{S_t}$ and preserves
irreversible. Since both the rank of $\tensor{\Pi}_2$ and
$\tensor{\Pi}_3$ are $2$, by Corollary~\ref{cor2}, we have that
%-------------------------------------------------------------------------------------------------------
\begin{eqnarray*}
 \Pr\set{\xi_t=i,\xi_{t+r}=j}=\Pr\set{\xi_t=j,\xi_{t+r}=i},\\
 \Pr\set{\eta_t=i,\eta_{t+r}=j}=\Pr\set{\eta_t=j,\eta_{t+r}=i}.
\end{eqnarray*}
%-------------------------------------------------------------------------------------------------------
In the case, one cannot detect irreversibility by comparing
directional moments $\Enum(\xi_t\xi^n_{t+r})$ and
$\Enum(\xi_t^n\xi_{t+r})$ with $n\in\Nnum$ like that used in
Reference~\cite[p104]{mz}.
\end{exmp}
%-------------------------------------------------------------------------------------------------------
\subsection{Proof} \label{subsec}
%-------------------------------------------------------------------------------------------------------
Let $\vec{e}_1=(1,0,0)',\vec{e}_2=(0,1,0)',\,\vec{e}_3=(0,0,1)'$,
$\tensor{E}_i=\diag\set{\vec{e}_i}$,
$\tensor{\wedge}_k=\diag\set{\varphi_{k}}$.
%-------------------------------------------------------------------------------------------------------

%-------------------------------------------------------------------------------------------------------
\begin{prop}\label{prop2}
If the underlying Markov process is irreversible and the rank of
the state-dependent probability is $3$, then the observed  process
is irreversible.
\end{prop}
%-------------------------------------------------------------------------------------------------------
\begin{proof}
Since the rank of the state-dependent probability is $3$, we can
choose a base of $\Rnum^3$, without loss generality, to be
$\varphi_{0},\varphi_{1},\varphi_{2}$. Then
$\vec{e}_1=\sum_{i=0}^2 x_i\varphi_i$, $\vec{e}_2=\sum_{j=0}^2
y_j\varphi_j$, and $\tensor{E}_1=\sum_{i=0}^2\tensor{\wedge}_i x_i
$, $\tensor{E}_2=\sum_{j=0}^2 \tensor{\wedge}_j y_j$.

Suppose on the contrary that the observed process is reversible.
When $t>0$,
%-------------------------------------------------------------------------------------------------------
\begin{eqnarray*}
0 &=& \Pr\set{S_0=i,S_t=j}-\Pr\set{S_0=j,S_t=i} \\
 &=& \mu \tensor{\wedge}_i\tensor{P}(t)\tensor{\wedge}_j\vec{e}-\mu
 \tensor{\wedge}_j\tensor{P}(t)\tensor{\wedge}_i\vec{e}, {\quad  \textup{where} \quad}i,j=0,1,2.
\end{eqnarray*}
%-------------------------------------------------------------------------------------------------------
Thus
%-------------------------------------------------------------------------------------------------------
\begin{eqnarray*}
  \mu_1 \tensor{P}_{12}(t)-\mu_2 \tensor{P}_{21}(t)
     &=& \mu \tensor{E}_1\tensor{P}(t)\tensor{E}_2\vec{e}-\mu
     \tensor{E}_2\tensor{P}(t)\tensor{E}_1\vec{e}\\
     &=&\sum\limits_{i,j=0}^2  x_i y_j
     \mu \tensor{\wedge}_i\tensor{P}(t)\tensor{\wedge}_j\vec{e}- \sum\limits_{i,j=0}^2 y_j x_i
     \mu \tensor{\wedge}_j\tensor{P}(t)\tensor{\wedge}_i\vec{e}\\
     &=&\sum\limits_{i,j=0}^2  x_i y_j[
     \mu \tensor{\wedge}_i\tensor{P}\tensor{\wedge}_j\vec{e}-
     \mu \tensor{\wedge}_j\tensor{P}\tensor{\wedge}_i\vec{e}]\\
     &=&0
\end{eqnarray*}
%-------------------------------------------------------------------------------------------------------
Note that
%-------------------------------------------------------------------------------------------------------
\begin{equation*}
  \lim\limits_{t\to
  0+}\frac{\tensor{P}(t)-\tensor{I}}{t}=\tensor{Q}.
\end{equation*}
%-------------------------------------------------------------------------------------------------------
Then $\nu=\mu_1 a_2-\mu_2 b_1=0$. However, the underlying Markov
process is irreversible, i.e., $\nu\neq 0$, a contradiction.
%-------------------------------------------------------------------------------------------------------
\end{proof}
%-------------------------------------------------------------------------------------------------------
%-------------------------------------------------------------------------------------------------------
\begin{lem}\label{cite3}
If the rank of $\tensor{\Pi}$ is $2$, then all its columns
$\set{\varphi_i, i=0,1,2,\dots,K-1}$ are the linear combination of its
certain column and $\vec{e}=(1,1,1)'$.
\end{lem}
%-------------------------------------------------------------------------------------------------------
\begin{proof}
If the rank is $2$, then $K\geqslant 2$. Without loss generality,
let $\set{\varphi_1,\,\varphi_2}$ be one base of $\set{\varphi_i,
i=0,1,2,\dots,K-1}$ with $\varphi_1\neq c\,\vec{e}$, where
$c\in\Rnum$. Note that
$\vec{e}=\sum_{i=0}^{K-1}\varphi_i=x\varphi_1+y\varphi_2$, where
$y\neq0,\,x\in \Rnum$, thus $\varphi_2=(\vec{e}-x\varphi_1 )/y$.
Since $\set{\varphi_1,\,\varphi_2}$ is one base, all
$\set{\varphi_i, i=0,1,2,\dots,K-1}$ are the linear combination of
$\set{\varphi_1,\,\vec{e}}$.
\end{proof}
%-------------------------------------------------------------------------------------------------------
Without loss generality, if the rank of $\tensor{\Pi}$ is $2$, let
$\set{\varphi_i, i=0,1,2,\dots,K-1}$ be the linear combination of
$\varphi_1$ and $\vec{e}$. Let $\wedge=\diag\set{\varphi_1}$.

%-------------------------------------------------------------------------------------------------------
\begin{prop}\label{prop11}
Suppose that the underlying Markov process is irreversible. If the
rank of $\tensor{\Pi}$ is $2$, and if any two rows of
$\tensor{\Pi}$ are not equal, then the observed process is
irreversible.
\end{prop}
%-------------------------------------------------------------------------------------------------------
\begin{proof}
We claim that $\varphi_1=(x,y,z)'$ with $x\neq y\neq z$.
Without loss of generality, suppose on the contrary that
$\varphi_1=(x,x,z)'$, then the first two rows of $\tensor{\Pi}$
are equal by Lemma~\ref{cite3}, a contradiction. Thus
$$D=(x-y)(y-z)(z-x)\neq 0.$$

Since the underlying Markov process is irreversible, the
transition rate flux $\nu \neq 0$. By Theorem~\ref{3dim likely}, when
$r,t>0$ and $r\neq t$,
$$\Pr\set{S_0=S_r=S_{r+t}=1} -\Pr\set{S_0=S_t=S_{t+r}=1}\neq 0,$$
i.e., the observed process is irreversible.
\end{proof}
%-------------------------------------------------------------------------------------------------------

%-------------------------------------------------------------------------------------------------------
\begin{prop}\label{prop3}
If there are two equal rows of $\tensor{\Pi}$, then the observed
process is time reversible.
\end{prop}
%-------------------------------------------------------------------------------------------------------
\begin{proof}
%-------------------------------------------------------------------------------------------------------
If the rank of $\tensor{\Pi}$ is $1$, i.e., all the three rows of
$\tensor{\Pi}$ are equal, the observed process is in fact
identical independent distribution series.

If the rank of $\tensor{\Pi}$ is $2$, without loss of generality,
suppose the first and the second row of $\tensor{\Pi}$ are equal.
Then $\varphi_{s_k}=x_k\vec{e}+y_{k}\vec{e}_3$ and
$\wedge_{s_k}=x_k\tensor{I}+y_k\tensor{E}_3$. The flux of the
likelihood function is
%-------------------------------------------------------------------------------------------------------
\begin{eqnarray*}
& &\Pr(S_{t_1}=s_1,\,S_{t_1+t_2}=s_2,\,\cdots,\,S_{t_1+t_2+\dots+t_r}=s_r)\\
&-&\Pr(S_{t_1}=s_r,\,S_{t_1+t_{r}}=s_{r-1},\cdots,\,S_{t_1+t_{r}+\dots+t_2}=s_1)\\
&=&
\mu\wedge_{s_{_1}}\tensor{P}(t_2)\wedge_{s_{_2}}\tensor{P}(t_3)\cdots\tensor{P}(t_r)\wedge_{s_{_r}}\vec{e}
-\mu\wedge_{s_{_r}}\tensor{P}(t_r)\wedge_{s_{_{r-1}}}\tensor{P}(t_{r-1})\cdots\tensor{P}(t_2)\wedge_{s_{_1}}\vec{e}\\
&=&\mu[x_{1}\tensor{I}+y_{1}\tensor{E}_3]\tensor{P}(t_2)[x_{2}\tensor{I}+y_{2}\tensor{E}_3]
\tensor{P}(t_3)\cdots\tensor{P}(t_r)[x_{r}\tensor{I}+y_{r}\tensor{E}_3]\vec{e}\\
&-&\mu[x_{r}\tensor{I}+y_{r}\tensor{E}_3]\tensor{P}(t_r)[x_{r-1}\tensor{I}+y_{r-1}\tensor{E}_3]
\tensor{P}(t_{r-1})\cdots\tensor{P}(t_2)[x_{1}\tensor{I}+y_{1}\tensor{E}_3]\vec{e}.
\end{eqnarray*}
%-------------------------------------------------------------------------------------------------------
Expand the expression, and delete the terms which do not contain
$\tensor{E}_3$. Note that $\tensor{E}_3^l=\tensor{E}_3$ for
$l\in\Nnum$(i.e., $\tensor{E}_3$ is projective matrix),
$\tensor{P}(t)\vec{e}=\vec{e}$, $\mu\tensor{P}(t)=\mu$, and
$\tensor{P}(t)\tensor{P}(r)=\tensor{P}(t+r)$ for $t,\,r\in
\Rnum^+$. All other terms pairwise satisfy that
%-------------------------------------------------------------------------------------------------------
\begin{eqnarray*}
 && \mu\cdots
\tensor{P}(t_{i})\tensor{I}\cdots\tensor{P}(t_{j})\tensor{E}_3\cdots\vec{e}-\mu\cdots
\tensor{E}_3\tensor{P}(t_{j})\cdots\tensor{IP}(t_{i})\cdots\vec{e}\\
  &=&\mu(\tensor{E}_3\tensor{P}(r)\tensor{E}_3)\cdots(\tensor{E}_3\tensor{P}(t)\tensor{E}_3)\vec{e}
  -\mu(\tensor{E}_3\tensor{P}(t)\tensor{E}_3)\cdots(\tensor{E}_3\tensor{P}(r)\tensor{E}_3)\vec{e}\\
  &= &\mu_3 \tensor{P}_{33}(r)\cdots\tensor{P}_{33}(t)-\mu_3
  \tensor{P}_{33}(t)\cdots\tensor{P}_{33}(r)\\
  &=&0.
\end{eqnarray*}
%-------------------------------------------------------------------------------------------------------
This ends the proof.
%-------------------------------------------------------------------------------------------------------
\end{proof}
%-------------------------------------------------------------------------------------------------------

%-------------------------------------------------------------------------------------------------------
\begin{prop}\label{mainthm}
Suppose that the underlying Markov process is irreversible.
%-------------------------------------------------------------------------------------------------------
\begin{itemize}
 \item[1)] If the state-dependent probability matrix $\tensor{\Pi}$
is singular, then the observed  process is reversible.
 \item[2)] If the state-dependent probability matrix $\tensor{\Pi}$
is regular, then the observed process is irreversible.
\end{itemize}
%-------------------------------------------------------------------------------------------------------
\end{prop}
%-------------------------------------------------------------------------------------------------------
\begin{proof}
The first case is Proposition~\ref{prop3}. If $\tensor{\Pi}$ is
regular, then the rank of $\tensor{\Pi}$ is $3$ or $2$. Thus the
second case is Proposition~\ref{prop11} and
Proposition~\ref{prop2}.
\end{proof}
%-------------------------------------------------------------------------------------------------------

%-------------------------------------------------------------------------------------------------------
%        proof of the main theorem
%-------------------------------------------------------------------------------------------------------
\noindent{\it Proof of Theorem \ref{dingl}.\,}
%-------------------------------------------------------------------------------------------------------
It can be shown directly by Theorem~\ref{kaishi} and
Proposition~\ref{mainthm}.
%-------------------------------------------------------------------------------------------------------
{\hfill\large{$\Box$}}
%-------------------------------------------------------------------------------------------------------

%-------------------------------------------------------------------------------------------------------
\section{Appendix: the discrete-time case}
Let $\set{S_t: \,t\in \Znum^+}$ be the observed process with state
space $\mathcal{S}=\set{0,1,2,\cdots,K-1}$.
%-------------------------------------------------------------------------------------------------------
Let $\set{C_t: \,t\in \Znum^+}$ be an irreducible three-state
Markov chain with the 1-step transition probability matrix
$\tensor{P}$ and the stationary distribution
$\mu=(\mu_1,\mu_2,\mu_3)$, where $\mu_1+\mu_2+\mu_3=1,\,\mu_i>0$.
%-------------------------------------------------------------------------------------------------------
\begin{equation}
\tensor{P}= \left[
\begin{array}{lll}
     1-a_2-a_3&a_2 &a_3\\
     b_1&1-b_1-b_3&b_3\\
     c_1&c_2 &1-c_1-c_2
\end{array}
\right ]
\end{equation}
%-------------------------------------------------------------------------------------------------------
By the stationarity, it is clear that the probability flux is
%-------------------------------------------------------------------------------------------------------
\begin{equation}
  \mu_1a_2-\mu_2b_1=\mu_2 b_3 - \mu_3 c_2 = \mu_3 c_1 - \mu_1 a_3.
\end{equation}
%-------------------------------------------------------------------------------------------------------
Let $\nu=\mu_1a_2-\mu_2b_1$. Let
$\tensor{Q}=\tensor{P}-\tensor{I}$, where $\tensor{I}$ is the unit
matrix. Denote by $-\lambda_1,\,-\lambda_2$ the
nonzero eigenvalues of $\tensor{Q}$.
%-------------------------------------------------------------------------------------------------------
\begin{lem}\label{sec4.1}
If $\triangle \neq 0$, then for $n\in \Nnum $, the $n$-step
transition probability matrix is
%-------------------------------------------------------------------------------------------------------
\begin{equation}
  \tensor{P}^n=g_n\tensor{L}+d_n\tensor{Q}+f_n\tensor{I},
\end{equation}
%-------------------------------------------------------------------------------------------------------
where
%-------------------------------------------------------------------------------------------------------
\begin{equation}
\begin{array}{ll}
d_n  = \frac{(1-\lambda_2)^n-(1-\lambda_1)^n}{\lambda_1-\lambda_2},\\
f_n =\frac
{\lambda_1(1-\lambda_2)^n-\lambda_2(1-\lambda_1)^n}{\lambda_1-\lambda_2},\\
g_n  = 1-f_n.
\end{array}
\end{equation}
%-------------------------------------------------------------------------------------------------------
\end{lem}
%-------------------------------------------------------------------------------------------------------
%-------------------------------------------------------------------------------------------------------
\begin{prop}\label{sec4.2}
The flux of the 2-dimension likelihood function is when
  $n\in\Nnum$,
%-------------------------------------------------------------------------------------------------------
\begin{equation*}
\hspace{-12mm} \Pr\set{S_0=i,S_n=j}-\Pr\set{S_0=j,S_n=i}=\frac{\nu A}{\lambda_1-\lambda_2}[(1-\lambda_2)^n-(1-\lambda_1)^n],
\end{equation*}
%-------------------------------------------------------------------------------------------------------
where $\nu$ is the probability flux,
$A=(y_2-x_2)(x_1-z_1)-(x_2-z_2)(y_1-x_1)$,
$(x_1,y_1,z_1)'=\varphi_i$, and $ (x_2,y_2,z_2)'=\varphi_j$.
\end{prop}
%-------------------------------------------------------------------------------------------------------
%-------------------------------------------------------------------------------------------------------
\begin{thm}\label{sec4.3}
The flux of the following 3-dimension likelihood function is when
  $n,m\in\Nnum$,
%-------------------------------------------------------------------------------------------------------
\begin{eqnarray*}
 & \Pr\set{S_0=S_n=S_{n+m}=i}- \Pr\set{S_0=S_m=S_{m+n}=i}\\
 &=\frac{\nu D}{\lambda_1-\lambda_2}[(1-\lambda _2)^n(1-\lambda
 _1)^m
  -(1-\lambda _1)^n(1-\lambda _2)^m],
\end{eqnarray*}
%-------------------------------------------------------------------------------------------------------
where $\nu$ is the probability flux, $D=(x-y)(y-z)(z-x)$,
$(x,y,z)'=\varphi_i$.
\end{thm}
%-------------------------------------------------------------------------------------------------------
%-------------------------------------------------------------------------------------------------------
\begin{thm}\label{dingl2}
The three-state hidden Makov models is irreversible, if and only
if the underlying Markov chain is irreversible, the state-dependent probability matrix
is regular, and one of the
following conditions is satisfied:
%-------------------------------------------------------------------------------------------------------
\begin{itemize}
    \item[1)] the rank of $\tensor{\Pi}$ is $3$,
    \item[2)] the rank of $\tensor{\Pi}$ is $2$, and $0$ is not the
eigenvalue of $\tensor{P}$.
\end{itemize}
%-------------------------------------------------------------------------------------------------------
\end{thm}
Proofs of Lemma~\ref{sec4.1}, Proposition~\ref{sec4.2} and Theorem~\ref{sec4.3} are omitted.
Proof of Theorem~\ref{dingl2} is presented in Subsection~\ref{proof2}.
%-------------------------------------------------------------------------------------------------------
\subsection{Proofs}\label{proof2}
%-------------------------------------------------------------------------------------------------------
\begin{lem}\label{if and only if}
If the rank of $\tensor{\Pi}$ is $2$, the observed process is
reversible if and only if for all $n,m,\dots,k\in\Nnum$,
%-------------------------------------------------------------------------------------------------------
\begin{equation}\label{if and only}
 \mu\tensor{\wedge}\tensor{P}^n\tensor{\wedge}\tensor{P}^m\tensor{\wedge}\cdots\tensor{\wedge}\tensor{P}^k\tensor{\wedge}\vec{e}
 -\mu\tensor{\wedge}\tensor{P}^k\tensor{\wedge}\cdots\tensor{\wedge}\tensor{P}^m\tensor{\wedge}\tensor{P}^n\tensor{\wedge}\vec{e}=0,
\end{equation}
%-------------------------------------------------------------------------------------------------------
where $\wedge=\diag\set{\varphi_1}$.
\end{lem}
%-------------------------------------------------------------------------------------------------------
\begin{proof}
The necessity is trivial. We need to prove the sufficiency only. It follows that $\varphi_{s_k}=x_k\vec{e}+y_{k}\varphi_1$
from Lemma~\ref{cite3}.
Then $\wedge_{s_k}=x_k\tensor{I}+y_{k}\tensor{\wedge}$. The flux
of the likelihood function is
%-------------------------------------------------------------------------------------------------------
\begin{eqnarray*}
& &\Pr(S_1=s_1,\,S_2=s_2,\,\cdots,\,S_l=s_l)-\Pr(S_l=s_1,\,S_{l-1}=s_2,\,\cdots,\,S_1=s_l)\\
&=&
\mu\wedge_{s_{_1}}\tensor{P}\wedge_{s_{_2}}\tensor{P}\cdots\wedge_{s_{_l}}\vec{e}
-\mu\wedge_{s_{_l}}\tensor{P}\wedge_{s_{_{l-1}}}\tensor{P}\cdots\wedge_{s_{_1}}\vec{e}\\
&=&\mu[x_{1}\tensor{I}+y_{1}\tensor{\wedge}]\tensor{P}[x_{2}\tensor{I}+y_{2}\tensor{\wedge}]
\tensor{P}\cdots[x_{l}\tensor{I}+y_{l}\tensor{\wedge}]\vec{e}\\
&-&\mu[x_{l}\tensor{I}+y_{l}\tensor{\wedge}]\tensor{P}[x_{l-1}\tensor{I}+y_{l-1}\tensor{\wedge}]
\tensor{P}\cdots[x_{1}\tensor{I}+y_{1}\tensor{\wedge}]\vec{e}\\
&=& \sum\limits_{\set{i_1,i_2,\cdots,i_s}} x_{i_1}\cdots
x_{i_s}y_{j_1}\cdots y_{j_k}\big[\mu\cdots
\tensor{PI}\cdots\tensor{P\wedge}\cdots\vec{e}-\mu\cdots
\tensor{\wedge P}\cdots\tensor{IP}\cdots\vec{e}\big].
\end{eqnarray*}
%-------------------------------------------------------------------------------------------------------
where $\set{i_1,i_2,\cdots,i_s}\in\set{1,2,\cdots,l}$ and
$\set{j_1,j_2,\cdots,j_k}=\set{1,2,\cdots,l}\setminus
\set{i_1,\cdots,i_s}$. Delete the term which does not contain
$\tensor{\wedge}$. Note that $\mu \tensor{P}=\mu$ and
$\tensor{P}\vec{e}=\vec{e}$.
%-------------------------------------------------------------------------------------------------------
\begin{eqnarray*}
&&\mu\cdots\tensor{PI}\cdots\tensor{P\wedge}\cdots\vec{e}-\mu\cdots
\tensor{\wedge P}\cdots\tensor{IP}\cdots\vec{e}\\
&=&\mu\tensor{\wedge}\tensor{P}^n\tensor{\wedge}\tensor{P}^m\tensor{\wedge}\cdots\tensor{\wedge}\tensor{P}^k\tensor{\wedge}\vec{e}
 -\mu\tensor{\wedge}\tensor{P}^k\tensor{\wedge}\cdots\tensor{\wedge}\tensor{P}^m\tensor{\wedge}\tensor{P}^n\tensor{\wedge}\vec{e}
\\
&=&0 .
\end{eqnarray*}
%-------------------------------------------------------------------------------------------------------
This ends the proof.
\end{proof}
%-------------------------------------------------------------------------------------------------------
%-------------------------------------------------------------------------------------------------------
\begin{rem}
Eq.(\ref{if and only}) is equivalent to for all positive integers
$r$ and all $0\leqslant t_1 \leqslant t_2 \leqslant
\cdots\leqslant t_r$,
%-------------------------------------------------------------------------------------------------------
\begin{equation}\label{equation 1}
 \Pr(S_{t_1}=S_{t_2}=\cdots=S_{t_r}=1)=\Pr(S_{t_1^-}=S_{t_2^-}=\cdots=S_{t_r^-}=1),
\end{equation}
%-------------------------------------------------------------------------------------------------------
where $t_l^-=t_1+t_r-t_l$.
\end{rem}
%-------------------------------------------------------------------------------------------------------
\begin{prop}\label{propzero}
If the rank of the state-dependent probability is $2$, and $0$ is
the eigenvalue of the 1-step transition probability, then the
observed process is reversible.
\end{prop}
%-------------------------------------------------------------------------------------------------------
\begin{proof}
Without loss generality, let $1-\lambda_2=0$. By
Lemma~\ref{probability},
 for all $n\in \Nnum$,
$$\tensor{P}^n=d_n\tensor{P}+g_n\tensor{L},$$
where $g_n=1-d_n,\,d_n=(1-\lambda_1)^{n-1}$.
%-------------------------------------------------------------------------------------------------------

%-------------------------------------------------------------------------------------------------------
\begin{eqnarray*}
&&
\mu\tensor{\wedge}\tensor{P}^n\tensor{\wedge}\tensor{P}^m\tensor{\wedge}\cdots\tensor{\wedge}\tensor{P}^k\tensor{\wedge}\vec{e}
 -\mu\tensor{\wedge}\tensor{P}^k\tensor{\wedge}\cdots\tensor{\wedge}\tensor{P}^m\tensor{\wedge}\tensor{P}^n\tensor{\wedge}\vec{e}\\
 &=&\mu\tensor{\wedge}[d_n\tensor{P}+g_n\tensor{L}]\tensor{\wedge}[d_m\tensor{P}+g_m\tensor{L}]\tensor{\wedge}\cdots\tensor{\wedge}[d_k\tensor{P}+g_k\tensor{L}]\tensor{\wedge}\vec{e}\\
 &-&\mu\tensor{\wedge}[d_k\tensor{P}+g_k\tensor{L}]\tensor{\wedge}\cdots\tensor{\wedge}[d_m\tensor{P}+g_m\tensor{L}]\tensor{\wedge}[d_n\tensor{P}+g_n\tensor{L}]\tensor{\wedge}\vec{e}\\
 &=&\sum\limits_{\set{i_1,i_2,\cdots,i_s}}
d_{i_1}\cdots d_{i_s}g_{j_1}\cdots
g_{j_k}\big[\mu\tensor{\wedge}\cdots
\tensor{P\wedge}\cdots\tensor{L\wedge}\cdots\tensor{\wedge}\vec{e}-\mu\tensor{\wedge}\cdots
\tensor{\wedge L}\cdots\tensor{\wedge
P}\cdots\tensor{\wedge}\vec{e}\big].
        \end{eqnarray*}
%-------------------------------------------------------------------------------------------------------
Delete the term which does not contain $\tensor{L}$. Note that
$\tensor{L}=\vec{e}\mu$.
%-------------------------------------------------------------------------------------------------------
\begin{eqnarray*}
  && \mu\tensor{\wedge}\cdots
\tensor{P\wedge}\cdots\tensor{L\wedge}\cdots\tensor{\wedge}\vec{e}\\
&=& (\mu\tensor{\wedge P}\cdots \tensor{
P\wedge}\vec{e})(\mu\tensor{\wedge
P}\cdots\tensor{P\wedge}\vec{e})
\cdots(\mu\tensor{\wedge P}\cdots\tensor{P\wedge}\vec{e})\\
  &=&\mu\tensor{\wedge}\cdots
\tensor{\wedge L}\cdots\tensor{\wedge
P}\cdots\tensor{\wedge}\vec{e}.
     \end{eqnarray*}
%-------------------------------------------------------------------------------------------------------
This ends the proof by Lemma~\ref{if and only if}.
\end{proof}
%-------------------------------------------------------------------------------------------------------
%-------------------------------------------------------------------------------------------------------
\begin{prop}\label{mainthm2}
Suppose that the underlying Markov chain is irreversible. Then we have
%-------------------------------------------------------------------------------------------------------
\begin{itemize}
 \item[1)] if there are two equal rows of $\tensor{\Pi}$, then the
observed  process is reversible;
 \item[2)] if the rank of $\tensor{\Pi}$ is $2$,\\
a) and if $0$ is the eigenvalue of $\tensor{P}$, then the
observed  process is reversible;\\
b) $\tensor{\Pi}$ is regular, and if $0$ is not the eigenvalue of $\tensor{P}$,
then the observed process is
irreversible;
    \item[3)]if the rank of $\tensor{\Pi}$ is $3$, then the observed  process is
    irreversible.
\end{itemize}
%-------------------------------------------------------------------------------------------------------
\end{prop}
%-------------------------------------------------------------------------------------------------------
\begin{proof}
The first case is similar to Proposition~\ref{prop3}. The second case is
Proposition~\ref{propzero} and similar to Proposition~\ref{prop11}. The third
case is similar to Proposition~\ref{prop2}.
\end{proof}
%-------------------------------------------------------------------------------------------------------

%-------------------------------------------------------------------------------------------------------
\noindent{\it Proof of Theorem \ref{dingl2}.\,}
%-------------------------------------------------------------------------------------------------------
It can be shown directly by Theorem~\ref{kaishi} and
Proposition~\ref{mainthm2}.
%-------------------------------------------------------------------------------------------------------
{\hfill\large{$\Box$}}
%-------------------------------------------------------------------------------------------------------
\section*{Acknowledgements}  This work is supported by Hunan Provincial Natural Science Foundation of China (No 10JJ6014).
%-------------------------------------------------------------------------------------------------------

%-------------------------------------------------------------------------------------------------------

\end{document}